\documentclass{amsart} 
\usepackage{amscd,amssymb,amsmath,amsbsy,amsthm, mathtools}
\usepackage{color,subfig}
\usepackage[colorlinks,plainpages,backref,urlcolor=blue]{hyperref}
\usepackage[width=5.65in,height=8.0in,centering]{geometry}
\usepackage[mathscr]{euscript}
\usepackage{comment}
\usepackage{bbm}
\usepackage[normalem]{ulem}
\usepackage{url}

\usepackage[table]{xcolor}

\usepackage{tikz}
\usetikzlibrary{calc,intersections,arrows,patterns,cd,decorations.markings,arrows.meta}
\usepackage{booktabs}
\usepackage{enumerate}
\usepackage{todonotes}
\usepackage{ulem}

\makeatletter
\newcommand{\mylabel}[2]{(#2)\def\@currentlabel{#2}\label{#1}}
\makeatother

\theoremstyle{definition}
\newtheorem{theorem}{Theorem}[section]
\newtheorem{definition}[theorem]{Definition}
\newtheorem{lemma}[theorem]{Lemma}
\newtheorem{corollary}[theorem]{Corollary}

\newtheorem{proposition}[theorem]{Proposition}

\theoremstyle{remark}
\newtheorem{remark}[theorem]{Remark}
\newtheorem{example}[theorem]{Example}

\newcounter{enumalpha}\renewcommand{\theenumalpha}{\alph{enumalpha}}

\makeatletter
\newtheorem*{rep@theorem}{\rep@title} \newcommand{\newreptheorem}[2]{%
\newenvironment{rep#1}[1]{%
\def\rep@title{\bf #2 \ref{##1}}%
\begin{rep@theorem} }%
{\end{rep@theorem} } }
\makeatother
\newreptheorem{theorem}{Theorem}
\newcommand{\abs}[1]{\left|#1\right|}               % absolute value delimiters
\newcommand{\set}[1]{\left\{#1\right\}}             % parentheses around a set
\newcommand{\angl}[1]{\left<#1\right>}              % inner product
\newcommand{\Z}{\mathbb{Z}}

\newcommand{\R}{\mathbb{R}}

\newcommand{\M}{\mathsf{M}}    % a matroid
\newcommand{\x}{\mathsf{x}}    % sequence of variables
\newcommand{\y}{\mathsf{y}}    % sequence of variables

\def\clap#1{\hbox to 0pt{\hss#1}}
\def\mathclap{\mathpalette\mathclapinternal}
\def\mathclapinternal#1#2{%
\clap{$\mathsurround=0pt#1{#2}$}}
\newcommand*{\dotleq}{\mathbin{\kern0.4em\cdot%
\mathclap{\raisebox{-0.05em}{$\leqslant$}}}}
\DeclareMathOperator{\rk}{rk}       % rank function

\DeclareMathOperator{\supp}{supp} % support

\DeclareMathOperator{\dom}{{dom}}          % domain

\DeclareMathOperator{\IN}{{IN}}

\newcommand{\1}{\mathbf{1}}
\newcommand{\HH}{\mathrm{H}}
\newcommand{\LL}{\mathrm{L}}
\def\strL{\mathring{\mathrm{L}}}

\newcommand{\Mnat}{$\text{M}^\natural$}

\numberwithin{equation}{section}
\begin{document}
%%%%%%%%%%%%%%%%%%%%%%%%%%%%%%%%%%%%%%%%%%%%%%%%%%%%%%%%%%%%%%%%%%%%%%%%%%%%%%%%

\title{Tree metrics and log-concavity for matroids}
\author{Federico Ardila--Mantilla} 
\address{%
  Queen Mary University of London and San Francisco State University}
   \email{federico@sfsu.edu}
%\and
  \author{Sergio Cristancho} 
\address{Princeton University} 
\email{sergio.cris@princeton.edu}
 % \and
\author{Graham Denham} %\footnote{%
 \address{University of Western Ontario} \email{gdenham@uwo.ca}
%\and
\author{\\ Christopher Eur} 
\address{Carnegie Mellon University} \email{ceur@cmu.edu}
%\and
\author{June Huh} %\footnote{%
\address{Princeton University and Korea Institute for Advanced Study} \email{huh@princeton.edu}
%\and
\author{Botong Wang} 
\address{University of Wisconsin-Madison} \email{wang@math.wisc.edu}
%}
%\date{}

\renewcommand{\shortauthors}{F.~Ardila--Mantilla, S.~Cristancho, G.~Denham, C.~Eur, J.~Huh, B.~Wang}

\begin{abstract}
We show that a set function $\nu$ 
  satisfies the gross substitutes property if and only if its homogeneous generating polynomial $Z_{q,\nu}$ is a Lorentzian polynomial  for all positive $q \le 1$, answering a question of Eur--Huh. We achieve this by giving a rank $1$ upper bound for the distance matrix of an ultrametric tree, refining a classical result of Graham--Pollak.
  This characterization enables us to resolve open problems that strengthen Mason’s log-concavity conjectures for the numbers of independent sets of a matroid:  one posed by Giansiracusa–Rinc\'on–Schleis–Ulirsch for valuated matroids, and two posed by Dowling in 1980 and Zhao in 1985 for ordinary matroids.
\end{abstract}

\maketitle

%%%%%%%%%%%%%%%%%%%%%%%%%%%%%%%%%%%%%%%%%%%%%%%%%%%%%%%%%%%%%%%%%%%%%%%%%%%%%%%%
\section{\textsf{Introduction}}

Let $E$ be a finite set of cardinality $n$. 
 A \emph{matroid} $\M$ on $E$ is given by a nonempty collection $\IN(\M)$ of subsets of $E$, called \emph{independent sets} of $\M$, satisfying the following properties:
\begin{enumerate}[(1)]\itemsep 5pt
\item If $I\in \IN(\M)$ and $J$ is a subset of $I$, then $J\in \IN(\M)$.
\item If $I,J\in \IN(\M)$ and $\abs{I} > \abs{J}$, then there is $i\in I\setminus J$ such that $J\cup i \in \IN(\M)$.
\end{enumerate}
The \emph{rank} of $\M$ is the common cardinality of the maximal independent sets. For background and any undefined terminology in matroid theory, we refer to \cite{OxleyMatroidTheory}.

Let $I_k=I_{k}(\M)$ be the number of
independent sets of $\M$ of cardinality $k$. 
Mason~\cite{Mason72}
conjectured that the following three families of inequalities hold for all $0<k<n$:
\begin{align}
  I_k^2&\geq I_{k-1}I_{k+1}, \tag{M1} \label{eq:weakMason}\\
  I_k^2 &\geq \Big(1+\frac1k\Big)I_{k-1}I_{k+1}, \tag{M2}
  \label{eq:mediumMason}\\
  I_k^2 & \geq \Big(1+\frac1k\Big)\Big(1+\frac1{n-k}\Big)I_{k-1}I_{k+1}. \tag{M3}\label{eq:strongMason}
\end{align}
In other words, the sequences $I_k$, $k!I_k$, and
$I_k/\binom{n}{k}$ are, respectively, \emph{log-concave} in $k$. The last property is referred to as the \emph{ultra log-concavity} of the sequence $I_0,\ldots,I_n$. Clearly, \eqref{eq:strongMason} implies \eqref{eq:mediumMason} implies \eqref{eq:weakMason} for a given matroid.

Conjecture \eqref{eq:weakMason} was proved for realizable matroids in \cite{Lenz13}, by building on the work of \cite{HK12}, and for arbitrary matroids in \cite{AHK18}. Conjecture \eqref{eq:mediumMason} was proved  in \cite{HSW22} using the main results of \cite{AHK18}, while Conjecture \eqref{eq:strongMason} was proved independently in~\cite{ALOVIII} and~\cite{BH20}. These inequalities motivated the development of combinatorial Hodge theory~\cite{AHK18} and the theory of Lorentzian polynomials~\cite{BH20}.

In Theorem~\ref{thm:valuated}, we extend \eqref{eq:strongMason} from matroids on $E$ to \Mnat-concave functions on $2^{E}$, which are precisely the set functions  
  satisfying the \emph{gross substitutes property} from economics. 
  This answers a question of  Giansiracusa--Rincón--Schleis--Ulirsch in \cite[Question 4.1]{GRSU24}, who asked whether \eqref{eq:strongMason} holds for the class of \Mnat-concave functions constructed from valuated matroids discussed in Example~\ref{ex:valuated}. 
  In Theorem~\ref{thm:qpolynomial},
  we prove a multivariate polynomial inequality that refines \eqref{eq:mediumMason}, affirmatively settling conjectures of Dowling~\cite{Do80} and Zhao~\cite{Zh85} that were recently popularized by Pak \cite{Pak24}.
Both results follow from our main result, Theorem~\ref{thm:qLorentzian}, characterizing \Mnat-concave functions in terms of Lorentzian polynomials:
\begin{quote}
A set function $\nu$ on $E$ satisfies the gross substitutes property if and only if its homogeneous generating polynomial $Z_{q,\nu} = \sum_{S \subseteq E} q^{-\nu(S)} x^S y^{\abs{E}-\abs{S}}$ is a Lorentzian polynomial for all positive $q \le 1$.
\end{quote}
This answers a question of Eur--Huh in \cite[Section 5.1]{EH20}, who asked for a characterization of set functions whose homogeneous generating polynomials are Lorentzian for all positive $q \le 1$.\footnote{The function $r$ from \cite[Example 5.4]{EH20} is  not \Mnat-concave, contrary to the authors' claim: the maximum of $r(\set{i,j})+r(\set{k})$ is achieved uniquely at $k=2$.}  
The key ingredient is Theorem~\ref{thm:psd}, which
gives a rank $1$ upper bound for the distance matrix of an ultrametric tree. This refines the classical result of Graham--Pollak that the distance matrix of a tree has exactly one positive eigenvalue.
In the rest of the introduction, we describe these results in detail.

\subsection{Inequalities for \Mnat-concave set functions}\label{sec:Mnat}

A  function $\nu \colon 2^{E} \to \R\cup\set{-\infty}$ is said to be \textit{\Mnat-concave} if it satisfies the \emph{exchange property}: For any $I_1, I_2 \subseteq E$ and $i_1 \in I_1\setminus I_2$, we have 
\begin{enumerate}[(1)]\itemsep 5pt
        \item either $\nu(I_1) + \nu(I_2) \leq \nu(I_1 \setminus i_1) +\nu(I_2 \cup i_1)$, or 
        \item there is $i_2\in I_2\setminus I_1$ such that $\nu(I_1)+\nu(I_2)\leq \nu((I_1\setminus i_1) \cup i_2)+\nu((I_2\setminus i_2) \cup i_1)$.
\end{enumerate}
The \emph{effective domain} of such a function is the set 
\[
\dom(\nu) \coloneqq \set{ I \subseteq {E} \, \colon \, \nu(I) \neq -\infty  }.
\]
For our purposes, without loss of generality, we may suppose that  $\dom(\nu)$ is nonempty.

Introduced by Murota and Shioura \cite{MS99}, \Mnat-concave functions are central objects in discrete convex analysis.
Fujishige--Yang \cite{FY03} and Reijnierse--van Gellekom--Potters \cite{Re02} independently proved that \Mnat-concavity is equivalent to the \emph{gross substitutes property} in economics, introduced by Kelso and Crawford \cite{KC82} two decades earlier.
For a comprehensive introduction to \Mnat-concave functions, we refer the reader to \cite[Chapter 6]{Mu03}. 
The notion of 
\Mnat-concave functions generalizes several concepts in matroid theory, including independent sets, rank functions, and valuated matroids.

\begin{example}[Matroid independent sets] \label{ex:ordinary}
	Let $\M$ be a matroid  on ${E}$. The \emph{independent set indicator function} of $\M$, defined by
    \[
	\nu_\M (S) \coloneqq \begin{cases}
		\quad 0 &\text{if $S$ is an independent subset of $\M$,}\\
		-\infty &\text{if $S$ is not an independent subset of $\M$,}
	\end{cases}
	\]
    is an \Mnat-concave function whose effective domain is
the collection of independent sets $\IN(\M)$.
\end{example}

\begin{example}[Matroid rank functions]\label{ex:rank}
The \emph{rank function} of $\M$, defined by
    \[
    \rk_\M(S) \coloneqq \max\Big\{\abs{I}, \ \text{$I$ is an independent set of $\M$ in $S$}\Big\},
    \]
    is an \Mnat-concave function whose effective domain  is  $2^{E}$. 
    More generally, any non-negative linear combination of the rank functions of the constituent matroids in a flag matroid is a \Mnat-concave function \cite[Theorem 3]{Sh12}. 
\end{example}

\begin{example}[Valuated matroid independent sets]\label{ex:valuated}
    A \emph{valuated matroid} of rank $d$ on $E$ is a function  $\underline{\nu}\colon {E\choose d}\to \R\cup\set{-\infty}$ 
    that satisfies the \emph{symmetric exchange property}: For any $d$-element subsets $B_1,B_2$ of $E$ and $b_1\in B_1\setminus B_2$, there is $b_2\in B_2\setminus B_1$ such that
    \[
    \underline{\nu}(B_1)+\underline{\nu}(B_2)\leq
    \underline{\nu}((B_1\setminus b_1)\cup b_2)+\underline{\nu}((B_2\setminus b_2)\cup b_1).
    \]
    Valuated matroids are precisely the possible height functions in a regular subdivision of a matroid polytope into matroid polytopes. \cite{Speyer}
 
Murota~\cite{Mu97} considered the  \Mnat-concave extension $ \nu$ of $\underline{\nu}$ to $2^{E}$ given by
\[
    \nu(S)\coloneqq \max\Big\{\underline{\nu}(B), \  \text{$B$ is a $d$-element subset of $E$ containing $S$}\Big\},
\]
where the maximum of the empty set is defined to be $-\infty$. The effective domain of $\nu$ is the collection of independent sets of a matroid on $E$, the \emph{underlying matroid} of $\underline{\nu}$.
\end{example}

For a function $\nu:2^{E} \to \mathbb{R} \cup \{-\infty\}$, an integer $0 \le k \le n$, and a  real number $0<q\leq 1$, we set
\[
I_{q,\nu;k}\coloneqq \sum_{I\in \binom{E}{k}} q^{-\nu(I)}, 
  \]
where, by convention, $q^\infty=0$. This sequence was first considered by Giansiracusa, Rincón, Schleis, and Ulirsch for valuated matroids in the context of Example~\ref{ex:valuated}, who proved in \cite[Theorem A]{GRSU24} that Murota's extension $\nu$ of a valuated matroid  satisfies 
a generalization of \eqref{eq:mediumMason}:
\[
  I_{q,\nu;k}^2 \geq \Big(1+\frac1k\Big)I_{q,\nu;k-1}  I_{q,\nu;k+1} \ \  \text{ for all $0<k<n$ and $0 < q \leq 1$}.
\]
They asked in \cite[Question 4.1]{GRSU24} whether $\nu$ satisfies a generalization of \eqref{eq:strongMason}, and they provided extensive numerical evidence in support of this.
We prove their prediction in the more general setting of \Mnat-concave functions.

\begin{theorem}\label{thm:valuated}
For any \Mnat-concave function $\nu\colon 2^{E} \to \mathbb{R} \cup \{-\infty\}$,  we have
\[
I_{q,\nu;k}^2 \geq \Big(1+\frac1k\Big)\Big(1+\frac1{n-k}\Big) I_{q,\nu;k-1}
    I_{q,\nu;k+1} \ \ \text{for  all $0<k<n$ and $0<q\leq 1$.}
\]
\end{theorem}

We deduce Theorem~\ref{thm:valuated} from an analytic characterization of 
 \Mnat-concave functions in terms of  \emph{Lorentzian polynomials} \cite{BH20}, whose definition we recall in Section~\ref{SectionLorentzian}.
We define the \emph{homogeneous generating polynomial}  a function $\nu\colon 2^{E} \to \mathbb{R} \cup \{-\infty\}$ by
\[
 Z_{q,\nu}(\x,y) \coloneqq \sum_{S\subseteq E}q^{-\nu(S)}\x^S y^{\abs{E \setminus S}} \ \ \text{with} \ \    \x^S \coloneqq \prod_{i \in S} x_i,
\]
where $\x=(x_i)_{i \in E}$ and $y$ is a  homogenizing variable different from $x_i$ for $i \in E$. We view $Z_{q,\nu}$ as a homogeneous polynomial of degree $n$ in $n+1$ variables $(\x,y)$ with a positive parameter $q$.

\begin{theorem}\label{thm:qLorentzian}
A function $\nu: 2^{E} \to \mathbb{R} \cup \{-\infty\}$ is \Mnat-concave if and only if its homogeneous generating polynomial
$Z_{q,\nu}$ is a Lorentzian polynomial for all positive $q \le 1$.
\end{theorem}

Theorem~\ref{thm:qLorentzian} answers a question of Eur--Huh in \cite[Section 5]{EH20}, who proved the ``if'' direction of Theorem~\ref{thm:qLorentzian} in \cite[Proposition 5.5]{EH20}. 
The other direction is a new contribution, which generalizes the following known results:
        \begin{enumerate}[(1)]\itemsep 5pt
            \item When $\nu=\nu_\M$ is the independent set indicator function of a matroid $\M$ in Example~\ref{ex:ordinary}, Theorem~\ref{thm:qLorentzian} recovers the statement that the \emph{homogeneous independent set generating polynomial}
            \[
            I(\x,y) \coloneqq \sum\limits_{I\in \IN(\M)} \x^I y^{\abs{E \setminus I}},
            \]
            is a Lorentzian polynomial \cite[Section 4.3]{BH20}. See \cite[Theorem 4.1]{ALOVIII} for an equivalent statement. 
            \item When $\nu=\rk_\M$ is the rank function of a matroid $\M$ as in Example~\ref{ex:rank}, Theorem~\ref{thm:qLorentzian} recovers the statement that the \emph{homogeneous multivariate Tutte polynomial} 
            \[
            T_{q}(\x,y)\coloneqq \sum\limits_{S\subseteq E} q^{-\rk_\M(S)} \x^S y^{\abs{E \setminus S}},
            \]
            is a Lorentzian polynomial for all positive $q \le 1$ in  \cite[Theorem 4.10]{BH20}. The Lorentzian property of the homogeneous independent set generating polynomial follows from the identity 
            \[
            I(\x,y)=\lim_{q \to 0} T_{q}(q \x,y).
            \]
        \end{enumerate}
Theorem~\ref{thm:valuated} follows from Theorem~\ref{thm:qLorentzian} by identifying the variables $x_i$ for $i \in E$, see Section~\ref{sec:genpoly}.
Mason's strongest conjecture \eqref{eq:strongMason} is the special case of Theorem~\ref{thm:valuated}  when $\nu=\nu_\M$.

\subsection{Polynomial inequalities for matroids}
Let $\M$ be a matroid on $E$, and let $\IN_k(\M)$ be the collection of $k$-element independent sets of $\M$. We consider the generating polynomial
\[
	I_k(\x)\coloneqq\sum_{S\in\IN(\M)_k} \x^S, \ \ \text{with} \ \  \x^S \coloneqq \prod_{i \in S} x_i,
    \]
where $\x=(x_i)_{i\in E}$. 
 For multivariate polynomials
$f,g\in \Z[x_i]_{i \in E}$, we write $f\succeq g$ if all the coefficients of the difference $f-g$ are nonnegative.
 %Pak \cite{Pak24}
 Dowling~\cite[Section 2]{Do80} conjectured that the polynomial refinement of Mason's conjecture~\eqref{eq:weakMason} holds for any matroid: that is, the inequality $ I_k(\x)^2 \succeq I_{k-1}(\x) I_{k+1}(\x)$ holds for all $k$, coefficient by coefficient.  Subsequently, Zhao~\cite{Zh85} made the stronger conjecture that the polynomial version of \eqref{eq:mediumMason} holds for all matroids.

 We prove Zhao's and Dowling's conjectures.
Using Theorem~\ref{thm:qLorentzian}, 
we establish a stronger polynomial inequality for the more general class of \Mnat-concave functions.
For a function $\nu:2^E \to \mathbb{R} \cup \{-\infty\}$, we set
\[
I_{q,\nu;k}(\x)\coloneqq \sum_{S\in\binom{E}{k}}q^{-\nu(S)}\x^S. 
\]
When $\nu$ is the independent set indicator function of $\M$,  the following theorem states that the polynomial version of \eqref{eq:mediumMason}, and hence also \eqref{eq:weakMason}, holds for $\M$. %This answers Pak's question on \eqref{eq:weakMason}.
%This implies Dowling's conjecture on \eqref{eq:weakMason} holds as well.

\begin{theorem}\label{thm:qpolynomial}
For any \Mnat-concave function $\nu:2^E \to \mathbb{R} \cup \{-\infty\}$, we have the coefficientwise inequality
  \[
  I_{q,\nu;k}(\x)^2 \succeq \Big(1+\frac1k\Big)I_{q,\nu;k-1}(\x) I_{q,\nu;k+1}(\x)
   \ \ \text{for  all $0<k<n$ and $0<q\leq 1$.}
  \]
\end{theorem}

In Section~\ref{sec:polynomial}, we prove the stronger\footnote{Unlike sequences of positive numbers, a sequence of positive polynomials can be locally log-concave but not globally log-concave coefficientwise; an example is the sequence $x, x+y, x+\frac74y, 3y$.}
statement that, for any \Mnat-concave function $\nu$ and any $0\leq i \leq j \leq k \leq l \leq n$ with $i+l=j+k$, we have 
 \[
   j! \, I_{q,\nu;j}(\x) \cdot k! \, I_{q,\nu;k}(\x) \succeq i! \, I_{q,\nu;i}(\x) \cdot  l! \, I_{q,\nu;l}(\x) \ \ \text{for all $0<q\leq 1$}.
  \]
For matroids, the displayed inequality admits the following combinatorial interpretation.
For a matroid $\M$ on $E$ and nonnegative integers $a$ and $b$ with $a+b=n$, we write $N_\M(a,b)$ for the number of partitions of $E =A \sqcup B$ into ordered independent sets $A$ and $B$ of $\M$ of sizes $a$ and $b$.

\begin{corollary}\label{cor:partition}
For any matroid $\M$ on $E$ and any  $0\leq i \leq j \leq k \leq l \leq n$ with $i+l=j+k$, 
\[
j! \, k! \, N_\M(j,k) \ge i! \, l! \, N_\M(i,l).
\]
\end{corollary}

As Pak pointed out, the polynomial version of Mason's strongest inequality \eqref{eq:strongMason} is not true: the sequence $I_k(\x)$ need not be  coefficient-wise ultra log-concave. For the uniform matroid of rank 2 on $\{1,2\}$, we have
\[
I_1(x_1,x_2)^2/2^2-I_0(x_1,x_2)I_2(x_1,x_2)=\frac{1}{4}(x_1^2+x_2^2-2x_1x_2)=\frac{1}{4}(x_1-x_2)^2,
\]
which is non-negative numerically, but not coefficient-wise. 
This illustrates the fact that Pak's polynomial refinement is significantly stronger than Mason's original conjecture.

\subsection{An inequality for ultrametric trees}
A central ingredient in the proof of Theorem~\ref{thm:qLorentzian} is a result on tree distance matrices that we now describe.
In their design of efficient address systems for communication networks, Graham and Pollak \cite{GP71} introduced the \emph{distance matrix} of a graph, whose $ij$ entry is the distance between vertices $i$ and $j$. They showed that the signature of this matrix gives a lower bound for the addresses in their design, and proved that the distance matrix of any tree has the Lorentzian signature $(+,-,\cdots,-)$. The latter fact also follows from work of Schoenberg~\cite[Theorem 1]{Sch35}, using the fact that an ultrametric tree can be metrically embedded in an $\ell^2$-space \cite{TV83}.  Our proofs of the matroid inequalities above rely on a refinement of this result for ultrametric trees.

An \emph{ultrametric tree} is a rooted tree with nonnegative lengths on its edges such that all the leaves are at the same distance from the root. If the tree has $n$ leaves, let $D$ be the $n \times n$ \emph{leaf distance matrix}, whose $ij$ entry
is the distance between the leaves $i$ and $j$. 
For $n \times n$ real symmetric matrices $A$ and $B$, we write $A \ge B$  if the difference $A-B$ is positive semidefinite.
We write $\1_{n\times n}$ for the $n \times n$ matrix all of whose entries are $1$.

\begin{theorem}\label{thm:psd} 
For any ultrametric tree with $n$ leaves whose common distance from the root is $1$, 
\[
\left(1 - \frac{1}{n} \right)  \1_{n\times n}  \ge \frac{1}{2}D.
\]
The constant $(1 - \frac{1}{n})$ is the smallest number that makes the above statement true.
\end{theorem}

Theorem \ref{thm:psd} can be seen as a quantitative refinement of Graham and Pollak's result that $D$ has at most one positive eigenvalue in the ultrametric case. 
In Section \ref{ssec:equality_cases}, we characterize the ultrametric trees for which the scalar $(1 - \frac1n)$ is optimal. 

Theorem~\ref{thm:psd} serves as the base case in our inductive argument for the proof of Theorem~\ref{thm:qLorentzian}.
This connection stems from Lemma~\ref{lem:ultra}, which states that any \Mnat-concave function on $2^E$ whose effective domain contains all the subsets of $E$ with at most $2$ elements gives rise to an ultrametric tree.
This is a variation of the standard fact that the space of uniform valuated matroids of rank 2 is equal to the space of phylogenetic trees \cite[Theorem 4.3.5]{MS15}. 

\subsection*{Acknowledgements}
The authors gratefully acknowledge the Institute for Advanced Study for providing an inspiring research environment.
Federico Ardila--Mantilla was supported by NSF Grant DMS 2154279, the Friends of the Institute for Advanced Study Fund, and a UK Royal Society Wolfson Fellowship. Graham Denham was supported by NSERC of Canada.
Chris Eur was supported by NSF Grant DMS 2246518.
Sergio Cristancho and June Huh were supported by the Simons Investigator Grant.
Botong Wang was supported by NSF Grant DMS 1926686.

Our results and proofs were first  announced in Feb. 2025 at the Institute for Advanced Study \cite{ArdilaIAStalk2, ArdilaIAStalk1}. In Jan. 2026, a few days after our work appeared on the arXiv, Cao, Chen, Li, and Wu \cite{CCLW} announced results that overlap with one of our four main theorems. We thank them for bringing Dowling and Zhao’s conjectures to our attention.

%%%%%%%%%%%%%%%%%%%%%%%%%%%%%%%%%%%%%%%%%%%%%%%%%%%%%%%%%%%%%%%%%%%%%%%%%%%%%%%%
\section{Trees}

\subsection{Ultrametric tree matrices} 

When $T$ is a tree with nonnegative edge lengths, let \(d(i,j)\) denote the
sum of the edge lengths along the unique path in $T$ connecting the vertices $i$ and $j$.
We call $d$ the \emph{tree distance function} for $T$, and we define the \emph{diameter}
of $T$ to be
$\displaystyle \max_{i,j} d(i,j)$.

If $T$ is rooted, the \emph{height} $H(i)$ of a vertex $i$ in $T$ is the
distance from $i$ to its furthest descendant in $T$. The term is chosen to match drawings of rooted trees
with the root at the top. For any distinct vertices $i$ and $j$ in $T$, let
$i \vee j$ denote their \emph{lowest common ancestor}, that is, the unique vertex of $T$ 
lying on all three paths connecting the root, $i$, and $j$ in $T$.

\begin{definition}\label{def:ultrametric}
  An \emph{ultrametric tree} $T$ is a rooted tree with a
  tree distance function for which 
  all leaves have the same distance from the root.  That
  distance is the \emph{radius} of the tree, which is also equal to the height of the root, and half the diameter.
\end{definition}

If $T$ is an ultrametric tree, the function $d$ satisfies the
\emph{ultrametric inequality}
\[
d(i,j)\leq \max\Big\{d(i,k),d(k,j)\Big\} \ \ \text{for any leaves $i,j,k$.}
\]
 It follows that the leaves of $T$ satisfies the \emph{three-point condition}:
\[
  \text{The maximum among $d(i,j),\, d(j,k),\, d(k,i)$ is attained at least twice for any leaves $i,j,k$.}
  \]
Every ultrametric tree with positive lengths defines an ultrametric on its leaves, and every finite ultrametric space arises in this way from an ultrametric tree~\cite[Theorem 3.1]{BG91}.

In this section, we prove the following extension of Theorem~\ref{thm:psd}.  
An \textit{upper subtree} $U$ of $T$ is an upper ideal of $T$ when regarded as a poset, that is, a set of vertices of $T$ such that every ancestor of a vertex in $U$ is also in $U$.

\begin{proposition}\label{prop:psd_general}
	Let $T$ be an ultrametric tree of radius $1$, let $U$ be a nonempty upper subtree of $T$, and let $M=U_{\text{min}}$ be the set of minimal elements of $U$. For each $i$ in $M$, let $n_i$ be the number of leaves of $T$ below $i$, and set $n\coloneqq\sum_{i\in M} n_i$. Then the symmetric matrix $A^{T,U}$ with rows and columns indexed by $M$ and with entries
	\[
	A_{ij}^{T,U} = \begin{cases}
		(1-\frac{1}{n}) - H(i\vee j), &\text{if $i \neq j$,}\\
		(1-\frac{1}{n}) - (1-\frac{1}{n_i}) H(i), &\text{if $i=j$,}
	\end{cases}
	\]
	is positive semidefinite. 
\end{proposition}

The first part of Theorem~\ref{thm:psd} is the special case when $U=T$.

\begin{proof}
	We may assume the tree $T$ is binary by replacing any vertex having $k>2$ children with $k-1$ vertices having two children each, connected by edges of length 0; this is illustrated in Figure~\ref{fig:binarize}.
	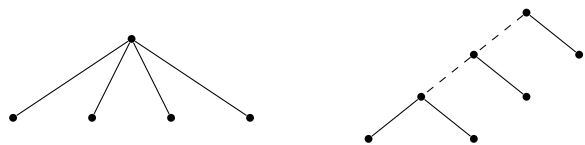
\begin{figure}[h]
	\begin{tikzpicture}[scale=0.7, transform shape,
		smallnode/.style = {circle, fill=black, inner sep=1.5pt},
		arrow/.style = {->, thick, >=Stealth}
		]
		
		% Binary tree
		\begin{scope}[xshift=7.5cm, level distance=0.8cm, sibling distance=2.5cm]
			
			\node[smallnode] (R1) at (0,0) {};
			\node[smallnode] (R2) at (-1,-0.8) {};
			\node[smallnode] (R3) at (-2,-1.6) {};
			\node[smallnode] (L1) at (-3,-2.4) {};
			\node[smallnode] (L2) at (-1,-2.4) {};
			\node[smallnode] (R4) at (0,-1.6) {};
			\node[smallnode] (R5) at (1,-0.8) {};
			
			% Edges
			\draw[dashed] (R1) -- (R2) node[midway, above] {};     % internal left edge
			\draw[dashed] (R2) -- (R3) node[midway, above] {};     % internal left edge
			\draw (R3) -- (L1);             % solid: to leaf
			\draw (R3) -- (L2);             % solid: to leaf
			\draw (R2) -- (R4);             % solid: to leaf
			\draw (R1) -- (R5);             % solid: to leaf
			
		\end{scope}
		
		% Original tree
		\node[smallnode] (A) at (0,-0.5) {}
		child {node[smallnode] {}}
		child {node[smallnode] {}}
		child {node[smallnode] {}}
		child {node[smallnode] {}};
		
	%	\draw[arrow] (2.5,-1.0) -- (4.5,-1.0);
		
	\end{tikzpicture}
        \caption{Non-binary trees can be seen as binary by adding edges of length $0$. \label{fig:binarize}}
	\end{figure}
	
	We prove this statement for all pairs $(T,U)$ by induction on $|V(T)|+|M|$. 
    Notice that when $|M|=1$, the matrix $A^{T,U}$ has a single entry which is nonnegative; this will be the base cases of our induction. When $U$ has $\abs{M} \geq 2$ leaves,  at least one of the following statements is true:
    \begin{enumerate}[(1)]\itemsep 5pt
        \item There is a leaf $i$ of $U$ with no siblings  in $U$.
        \item There are sibling leaves $i,j$ of $U$ with a common parent $h=i\vee j$.
    \end{enumerate}
    
In the first case, let $h$ and $g$ be the parent and grandparent of $i$, respectively; they must both exist since $\abs{M} \geq 2$. Let $T'$ be the smaller tree obtained by removing $h$ and all the descendants of $h$ that are neither $i$ nor descendants of $i$, and replacing edges $(g,h)$ and $(h,i)$ by an edge $(g,i)$ of length $\ell(g,h)+\ell(h,i)$. 
If $U'$ is the corresponding upper subtree of $T'$,  
then  $A^{T,U}=A^{T',U'}$ and the induction hypothesis applies to $(T',U')$.  

In the second case, let $U'$ be the upper subtree $U\backslash\set{i,j}$ of $T$, whose set of minimal elements  $M'=M\backslash\set{i,j} \cup \set{h}$ is smaller. We compare the matrices $A^{T,U}$ and $A^{T,U'}$. 
Since $T$ is binary, $h$ has no other children other than $i$ and $j$, so
%\begin{equation}\label{eq:vee}
\[
i\vee k=j\vee k=h\vee k \ \  \text{for any $k\in M \backslash\set{i,j}$}.
\]
%\end{equation}
Setting  $\langle a\rangle \coloneqq 1- \frac{1}{a}$ and focusing on the $i$-th and $j$-th rows and columns of the matrix $A^{T,U}$, which are shown as the first and the second rows and columns below,  we have
\begin{eqnarray*}
		A^{T,U} 
		&=& 
		\begin{bmatrix}
			& %\cellcolor{cyan!25}
            \angl{n} - \angl{n_i} H(i) &  \angl{n} - H(h) & \angl{n} - H(i \vee k) & \cdots  \\
			& \angl{n} - H(h) &  \angl{n} - \angl{n_j}H(j) & \angl{n} - H(j \vee k) & \cdots \\
			& \angl{n} - H(i \vee k) & \angl{n} - H(j \vee k) & \angl{n} - \angl{n_k}H(k)  & \cdots \\
			& \vdots & \vdots & \vdots & \ddots& 
		\end{bmatrix}
        \text{ since } i \vee j = h
        \\
		&\geq &
		\begin{bmatrix}
			& \angl{n} - \angl{n_i}H(h) &  \angl{n} - H(h) & \angl{n} - H(i \vee k) & \cdots  \\
			&\angl{n} - H(h) & \angl{n} - \angl{n_j}H(h) & \angl{n} - H(j \vee k) & \cdots \\
			& \angl{n} - H(i \vee k) & \angl{n} - H(j \vee k) & \angl{n} - \angl{n_k}H(k)  & \cdots\\
			& \vdots & \vdots & \vdots & \ddots& 
		\end{bmatrix}
		\text{ since } H(i), H(j) \leq H(h)
		\\
		&\sim &
		\begin{bmatrix}
			&  (2-\angl{n_i}-\angl{n_j})H(h) &(\angl{n_j}-1)H(h) & 0 & \cdots  \\
			&  (\angl{n_j}-1)H(h)&  \angl{n} - \angl{n_j}H(h) & \angl{n} - H(h \vee k) & \cdots \\
			& 0 & \angl{n} - H(h \vee k) & \angl{n} - \angl{n_k}H(k)  & \cdots \\
			& \vdots & \vdots & \vdots & \ddots& 
		\end{bmatrix}
		{\setlength\arraycolsep{1.5pt}
		\begin{aligned} 
            &\text{replacing} \\
			&\text{row } i \mapsto  \text{row } i - \text{row } j \\
			&\text{col } i  \mapsto   \text{col } i - \text{col } j, \\
            &\text{using $i\vee k=j\vee k=h\vee k$.} %\eqref{eq:vee}}
		\end{aligned}}
		\\
		&\sim &
		\begin{bmatrix}
			& %\cellcolor{cyan!25} 
            (2-\angl{n_i}-\angl{n_j})H(h) &  0 & 0 & \cdots  \\
			&  0 &  \angl{n} - \angl{n_h}H(h) & \angl{n} - H(h \vee k) & \cdots \\
			& 0 & \angl{n} - H(h \vee k) & \angl{n} - \angl{n_k}H(k)  & \cdots \\
			& \vdots & \vdots & \vdots & \ddots& 
		\end{bmatrix}
		{\setlength\arraycolsep{1.5pt}
		\begin{aligned}
			&\text{replacing} \\
            &\text{row } j \mapsto  \text{row } j + \frac{n_i}{n_h}\text{row } i \\
			&\text{col } j \mapsto  \text{col } j + \frac{n_i}{n_h}  \text{col } i.
		\end{aligned}}
		\\
		&= &
		 \left[
\begingroup
\renewcommand*{\arraystretch}{1.3}
  \begin{array}{c|c}
     (\frac{1}{n_i}+\frac{1}{n_j})H(h) & 0\\ \hline
    0 & A^{T,U'}\end{array}\endgroup
  \right].
	\end{eqnarray*}
Therefore, $A^{T,U}$ is positive semidefinite by the induction hypothesis. 
The third step creates zeros in the $(i,j)$-th and $(j,i)$-th entries because
	\[
	\frac{n_i}{n_h} = \frac{n_i}{n_i+n_j}=\frac{1-\angl{n_j}}{2-\angl{n_i}-\angl{n_j}}.
	\]
The formula for the $(j,j)$-th entry follows from the following identity for $a=n_i$ and $b=n_j$:
	\begin{equation}\label{eq:<n>}
		\angl{a+b}= \frac{1-\angl{a}\angl{b}}{2-\angl{a}-\angl{b}}. \tag{$*$}
	\end{equation}
	This last step features a pleasant subtlety worth mentioning explicitly: our proof requires a function $\angl{ \cdot } \colon \Z_{>0} \rightarrow \R_{\geq 0}$ that satisfies \eqref{eq:<n>}. A priori, this functional equation would seem overdetermined, since it gives different formulas for $\angl{a+b}$ for different choices of $a$ and $b$. In fact, the solutions to \eqref{eq:<n>} are $\langle x \rangle = 1-\frac{1}{cx}$ for any nonzero $c$, and the smallest positive $c$ such that $1-\frac{1}{cx}$ is nonnegative for all positive $x$ is $1$.
	Thus, the choice of constants $\angl{n_k} = 1 - \frac1{n_k}$ in the matrix $A^{T,U}$
     is exactly what ensures that the inductive machinery operates smoothly.
\end{proof}

\begin{reptheorem}{thm:psd} 
For any ultrametric tree with $n$ leaves whose common distance from the root is $1$, 
\[
\left(1 - \frac{1}{n} \right)  \1_{n\times n}  \ge \frac{1}{2}D.
\]
The constant $(1 - \frac{1}{n})$ is the smallest number that makes the above statement true.
\end{reptheorem}

\begin{proof}
The desired inequality is the statement that $A^{T,T}$ is positive semidefinite. For the second statement notice that the star tree, where the root is the direct parent of all the leaves, has $\frac12D=\1_{n\times n} - I_{n \times n}$. Therefore $\lambda \1_{n \times n} - \frac12 D = I_{n \times n} - (1-\lambda) \1_{n \times n}$, whose eigenvalues are $1, \ldots, 1, 1-(1-\lambda)n$. The smallest $\lambda$ that makes this matrix positive semidefinite is $1 - \frac1n$, as desired.
\end{proof}

\subsection{Equality cases} \label{ssec:equality_cases}

For an ultrametric tree $T$ of radius 1 with $n$ leaves and a nonnegative real
number $c$, we consider the $n\times n$ matrix 
$A(c) = c \, \1_{n\times n} - \frac{1}{2}D.$
The matrix $A(0)$ is not positive semidefinite, while the matrix $A({1-\frac{1}{n}})$ is positive semidefinite by Theorem \ref{thm:psd}.  Since the cone of positive semidefinite matrices is closed, there is a smallest positive constant $c=c_T\leq 1-\frac{1}{n}$ for which the matrix $A(c)$ is positive semidefinite.
When do we have $c_T =1-\frac{1}{n}$?

We saw in the proof of Theorem \ref{thm:psd} that star trees are optimal in this sense. 
It turns out that
they are essentially the only possibility for $c_T = 1-\frac1n$.   In order to
formulate the precise statement, we introduce some auxiliary definitions for ultrametric trees of radius $1$. Since we allow edges with length zero,
the distance function on $T$ may define only a
pseudometric on the leaves. In other words, we may have $d(i,j)=0$ for distinct $i$ and $j$.
\begin{enumerate}[(1)]\itemsep 5pt
    \item We say that $T$ is \textit{leaf-positive} if every edge of $T$ adjacent to a leaf has positive length. Note that this is equivalent to saying that for any vertex $i$ of $T$, $H(i)=0$ if and only if $i$ is a leaf. 
    \item We say that $T$ is a \textit{star-metric} if, for any two leaves $i,j$ of $T$, the height of their lowest common ancestor $H(i\vee j)$ is either 0 or 1. 
\end{enumerate}
The leaf-positive condition implies that the tree distance function is a metric.
The star-metric condition characterizes those ultrametric trees
that give star trees after contracting all length 0 edges and suppressing all non-root degree 2 vertices. 

\begin{proposition}\label{prop:equality_psd}[Equality case of Theorem \ref{thm:psd}] 
	Let $T$ be a leaf-positive ultrametric tree of radius 1.  Then $c_T=1-\frac1n$ if and only if $T$ is a star-metric. 
\end{proposition}

The proof is obtained by carefully retracing the proof of Proposition~\ref{prop:psd_general}.\footnote{For details see  \url{http://sergiocs147.github.io/files/ACDEHW-equality.pdf}.} 
Proposition ~\ref{prop:equality_psd}
will not be used in  the rest of the paper.

\section{Lorentzian polynomials from trees}

\subsection{Lorentzian polynomials and discrete convex analysis}\label{SectionLorentzian}

Let $\HH^d_n$ be the space of degree $d$ homogeneous polynomials in $n$
variables $\x=(x_i)_{i \in E}$ with real coefficients, equipped with the usual topology of a finite-dimensional real vector space. 
We write $\partial_i$ for the differential operator $\frac{\partial}{\partial x_i}$ for $i\in E$, and set
\[
\partial^\alpha \coloneqq \prod_{i\in E} \partial_i^{\alpha_i}  \ \ \text{for any $\alpha=(\alpha_i)_{i \in E}\in \Z^E_{\geq 0}$}.
\]
In degree $2$, the set of \emph{strictly Lorentzian polynomials} $\strL_n^2\subseteq \HH^2_n$
is, by definition, the set of quadratic forms $f$ whose Hessian is entrywise positive and has 
Lorentzian signature: 
\[
\strL^2_n \coloneqq \set{f\in \HH^2_n  \;\middle|\; \left(\partial_i \partial_j f\right)_{i,j \in E} \text{~has only positive entries and  signature}  (+, - \cdots, -)}.
\]
 In degree $d>2$, we
define the set of \emph{strictly Lorentzian polynomials} $\strL_n^d\subseteq \HH^d_n$ recursively by setting
\[
\strL^d_n \coloneqq \set{f\in \HH^d_n  \;\middle|\;  \partial_i f\in \strL_n^{d-1}\text{~for all
    $i \in E$}}.
\]
The set of  \emph{Lorentzian polynomials} $\LL_n^d$
is defined to be the closure of $\strL_n^d$ in $\HH^d_n$. 
Here we collect some properties of Lorentzian polynomials relevant to this paper: 

\begin{enumerate}[(1)]\itemsep 5pt
    \item If $f(\x)\in \LL^d_n$ and $A$ is an $n\times m$ nonnegative matrix, then $f(A\y) \in \LL^d_m$ \cite[Theorem 2.10]{BH20}.
    \item If $f\in \LL_n^d$, then $\sum_i a_i \partial_i f \in \LL_n^{d-1}$ for any $a_i \ge 0$ \cite[Corollary 2.11]{BH20}.
    \item If $f\in \LL_n^d$ and $g\in \LL_n^e$, then $fg\in \LL^{d+e}_n$ \cite[Corollary 2.32]{BH20}.
    \item A degree $d$ bivariate polynomial $f(x,y) =  \sum\limits_{k=0}^d a_k x^ky^{d-k}$ 
is Lorentzian if and only if 
\[
(\partial_x^d f, \ldots,\partial_x^{d-k}\partial_y^{k}f,\ldots,\partial_y^d f) = 
\left(\frac{a_0}{\binom{d}{0}}, \ldots, \frac{a_k}{\binom{d}{k}}, \ldots \frac{a_d}{\binom{d}{d}}\right)
\]
is a log-concave sequence of nonnegative numbers with no internal zeros \cite[Example 2.3]{BH20}. 
\end{enumerate}

We briefly review the theory of Lorentzian polynomials and its relation to  discrete convex analysis.
For further details, we refer to \cite{BH20} for Lorentzian polynomials and \cite{Mu03} for discrete convex analysis. 
We consider 
 the rank $d$ simplex in $\Z_{\ge 0}^E$ defined by
\[
\Delta_n^d \coloneqq  \set{\alpha \in \Z_{\geq0} ^{E} \;\middle|\;  \alpha_1 + \dotsb + \alpha_n = d}.
\] 
For $i \in E$, we write $e_i$ for the $i$-th standard basis vector of $\Z^{E}$. 
An \emph{M-convex set} of rank $d$ on $E$ is a subset $\mathrm{J}\subseteq \Delta_n^d$ satisfying the \emph{symmetric exchange property}: 
\begin{quote}
For any $\alpha, \beta \in \mathrm{J}$ and $i \in E$ such that $\alpha_i > \beta_i$, there is $j \in E$ such that $\beta_j >\alpha_j$ for which both $\alpha + e_j -e_i$ and $\beta +e_i-e_j$ belong to $\mathrm{J}$. 
\end{quote}
A function $\nu\colon \Delta_n^d \to \R \cup \{\infty\}$ is said to be \emph{M-convex} if, 
for any $\alpha, \beta \in \Delta_n^d$ and $i\in E$ such that $\alpha_i > \beta_i$, there is $j\in E$ such that $\beta_j >\alpha_j$ and
\[
\nu(\alpha) + \nu(\beta)\ge \nu(\alpha + e_j -e_i) + \nu(\beta +e_i-e_j).
\]
Equivalently, a set of lattice points is $M$-convex if and only if it is the set of lattice points in an integral generalized permutahedron, and a function is $M$-convex if and only if it induces a regular subdivision of an integral generalized permutahedron into integral generalized permutahedra.

We say that $\nu$ is \emph{M-concave} if $-\nu$ is M-convex.
The \emph{effective domain} of an M-convex function is 
\[
\dom(\nu)\coloneqq \set{\alpha\in \Delta_n^d  \;\middle|\;   \nu(\alpha) \neq \infty},
\]
and the effective domain of an M-concave function is defined similarly. 
Note that the effective domains of M-convex functions and M-concave functions are M-convex.
The \emph{support} of  a degree $d$ homogeneous polynomial in $n$ variables $f=\sum_{\alpha \in \Delta_n^{d}} c_\alpha \x^\alpha$ is defined by
\[
\supp(f) \coloneqq \set{ \alpha \in \Delta_n^{d} \;\middle|\; c_\alpha \neq 0 }.
\]
The following characterization of Lorentzian polynomials is central to their relationship to matroids and M-convexity \cite[Theorem 2.25]{BH20}.

\begin{theorem}\label{thm:Lorentzian}
The following conditions are equivalent for $f \in \HH^d_n$:
    \begin{enumerate}[(1)]\itemsep 5pt
        \item The polynomial $f$ is Lorentzian.
        \item The support of $f$ is an M-convex set and, for all $\alpha \in \Delta_n^{d-2}$, the Hessian of ${\partial^\alpha f}$ has only nonnegative entries and has at most one positive eigenvalue.
    \end{enumerate}
\end{theorem}

From this, one may deduce the following characterizations of M-convex functions \cite[Theorem 3.14]{BH20}: A function $\nu:\Delta^d_n \to \mathbb{R} \cup \{\infty\}$ is M-convex if and only if its \emph{normalized generating polynomial}
\[
f_{q,\nu}(\x) \coloneqq \sum_{\alpha \in \dom(\nu)} q^{\nu(\alpha)}\frac{\x^\alpha}{\alpha!}
\]
is a Lorentzian polynomial for all positive $q \le 1$, where 
${\x^\alpha}/{\alpha!}$ is the product of ${x_i^{\alpha_i}}/{\alpha_i!}$ for $i\in E$. 
This in turn implies the following characterization of M-convex sets \cite[Theorem 3.10]{BH20}: A subset $\mathrm{J}$ of $\Delta^d_n$ is M-convex if and only if its normalized generating polynomial
\[
f_\mathrm{J} \coloneqq \sum_{\alpha \in \mathrm{J}}\frac{\x^\alpha}{\alpha!}
\]
is a Lorentzian polynomial. For example, a collection of $d$-element subsets of $E$ is the set of bases of a matroid on $E$ if and only if its generating polynomial is Lorentzian.

\begin{remark}
The above characterization of M-convex functions  specializes to the following characterization of \Mnat-concave functions on $2^E$: A function $\nu:2^E \to \mathbb{R} \cup\{-\infty\}$ is \Mnat-concave if and only if the \emph{normalized homogeneous generating polynomial}
\[
N(Z_{q,\nu})\coloneqq \sum_{S \subseteq E} q^{-\nu(S)} \x^S \frac{y^{\abs{E \setminus S}}}{\abs{E \setminus S}!}
\]
is a Lorentzian polynomial for all positive $q \le 1$, where $y$ is a homogenizing variable different from $x_i$ for $i \in E$. 
Identifying the variables $x_i$ with each other, this characterization gives the property (\ref{eq:mediumMason}) for $I_{q,\nu;k}$.
It is interesting to compare this characterization of \Mnat-concave functions on $2^E$ with that in Theorem~\ref{thm:qLorentzian}, which implies  the property (\ref{eq:strongMason}) for $I_{q,\nu;k}$. 
Note that, in general, if we omit any normalizing factor $\frac{1}{\alpha_i!}$ from the normalized generating polynomial of an M-convex function, we do not get a Lorentzian polynomial. For example, among the three bivariate quadratic forms
\[
\frac{x_1^2}{2} +x_1x_2+\frac{x_2^2}{2}, \quad
x_1^2 +x_1x_2+\frac{x_2^2}{2}, \quad
x_1^2 +x_1x_2+x_2^2, 
\]
only the first is a Lorentzian polynomial.
\end{remark}

\subsection{Lorentzian property of the homogeneous generating polynomial $Z_{q,\nu}$}\label{sec:genpoly}
The following lemma is a variation of known relations between tree metrics, ultrametrics, and valuated matroids. See, for instance, \cite[Section 4.3]{MS15}.

\begin{lemma}\label{lem:ultra}
Let $\nu$ be an \Mnat-concave function on $2^E$ whose effective domain contains all subsets of $E$ with at most $2$ elements. Then,  for any $0<q\leq 1$, the function
  \[
  d(i,j)\coloneqq 
  \begin{cases}
  2q^{-\nu(i,j)+\nu(i)+\nu(j)-\nu(\varnothing)}, & \text{if $i\neq j$,}\\
  0, & \text{if $i=j$,}
  \end{cases}
  \]
  defines an ultrametric on $E$ of radius $\le 1$. %of radius $\leq 1$.  %of radius 1
\end{lemma}
\begin{proof}
  By the local exchange property for \Mnat-concave functions \cite[Theorem 6.4]{Mu03}, for any distinct elements $i,j,k$ in $E$, the maximum among the three numbers
  \[
    \nu(j,k)+\nu(i), \quad \nu(i,k)+\nu(j), \quad \nu(i,j)+\nu(k),
  \]
  is achieved at least twice. Thus, the maximum among the three numbers
  \[
    \nu(j,k) -\nu(j)-\nu(k) + \nu(\varnothing)\, \quad \nu(i,k) -\nu(i)-\nu(k) + \nu(\varnothing), \quad \nu(i,j) -\nu(i)-\nu(j) + \nu(\varnothing),
  \]
  is achieved at least twice as well. Since $0 < q \leq 1$, we get that $d$ is an ultrametric on $E$. 
  Also, the exchange property for $\nu$ gives
  \[
\nu(i,j)+\nu(\varnothing) \le \nu(i)+\nu(j) \ \ \text{for any distinct $i, j$ in $E$.}
  \]
which implies that $d$ has radius $\le 1$.
\end{proof}

We now prove Theorem~\ref{thm:qLorentzian} and deduce Theorem~\ref{thm:valuated}.
\begin{reptheorem}{thm:qLorentzian}
A function $\nu: 2^{E} \to \mathbb{R} \cup \{-\infty\}$ is \Mnat-concave if and only if its homogeneous generating polynomial
$Z_{q,\nu}$ is a Lorentzian polynomial for all positive $q \le 1$.
\end{reptheorem}

\begin{proof}
The ``if'' direction is \cite[Proposition 5.5]{EH20}. We show that $Z_{q,\nu}$ is Lorentzian for positive $q \le 1$ when $\nu$ is \Mnat-concave.

We observe that any \Mnat-concave function on $2^E$ can be approximated by a sequence of \Mnat-concave functions whose effective domains are $2^E$. More precisely, if $\nu$ is such an \Mnat-concave function, then there is a sequence \Mnat-concave functions $\nu_k$ with $\dom (\nu_k)=2^{E}$ such that 
\[
\lim_{k\to \infty} \nu_k(S) = \nu(S) \ \ \text{for every $S \subseteq E$.}
\]
This is a special case of~\cite[Lemma 3.27]{BH20}, because 
restrictions of M-concave functions to coordinate half-spaces are M-concave \cite[Section 6.4]{Mu03}.
 Since $\mathrm{L}^d_n$ is closed, we may assume that $0<q<1$, and for any such $q$, we have
\[
 \lim\limits_{k\to \infty} Z_{q,\nu_k} = Z_{q,\nu}.
\]
Therefore, we may assume that $0<q<1$ and the effective domain of $\nu$ is $2^{E}$.

Since the support of $Z_{q,\nu}$ is M-convex, by Theorem~\ref{thm:Lorentzian}, it is enough to 
show that all the $n+1$ partial derivatives of $Z_{q,\nu}$ are Lorentzian.
For any $i \in E$, we have
\[
\frac{\partial}{\partial x_i}Z_{q,\nu}=Z_{q,\nu/i},
\]
where $\nu/i$ is the \Mnat-concave function on $2^{E \setminus i}$ defined by $\nu /i (S)=\nu(S \cup i)$. For general discussion about the \emph{contraction} $\nu /i$ of an \Mnat-concave function $\nu$, see \cite[Section 6.4]{Mu03}. Thus, modulo induction on $n$, we only need to consider the Lorentzian property of the quadratic form
  \[
    \left(\frac{\partial}{\partial y}\right)^{n-2}Z_{q,\nu} 
      =(n-2)!\left[ \frac{n(n-1)}{2} q^{-\nu(\varnothing)} y^2+(n-1)\sum_{i \in \binom{E}{1}}q^{-\nu(i)}
        x_i y+\sum_{ij \in \binom{E}{2}}q^{-\nu(i,j)}x_ix_j \right].
        \]
  Up to positive rescaling of rows and columns, the Hessian of this quadratic form is 
  \[
  A(\nu)\coloneqq \left[
\begingroup
\renewcommand*{\arraystretch}{1.3}
  \begin{array}{c|ccc}
    n/(n-1)q^{-\nu(\varnothing)} & q^{-\nu(1)} & \cdots & q^{-\nu(n)}\\ \hline
    q^{-\nu(1)} & 0 & \cdots & q^{-\nu(i,j)}\\
      \vdots & & \ddots & \\
      q^{-\nu(n)} & q^{-\nu(i,j)} & & 0
  \end{array}
\endgroup
  \right].
  \]
Our goal is to show that $A(\nu)$ has at most one positive eigenvalue for any positive $q \le 1$.
  Computing the Schur complement with respect to the block structure indicated above, we may block diagonalize $A(\nu)$ to
  \[
  \left[
\begingroup
\renewcommand*{\arraystretch}{1.3}
  \begin{array}{c|c}
    n/(n-1)q^{-\nu(\varnothing)} & 0\\ \hline
    0 & B(\nu)\end{array}\endgroup
  \right],
  \]
  where $B(\nu)$ is the symmetric matrix with rows and columns labeled by $E$ and with entries
  \[
  B(\nu)_{ij} = \begin{cases} \frac{1-n}n
    q^{\nu(\varnothing)-\nu(i)-\nu(j)}+q^{-\nu(ij)}, & \text{if $i\neq j$,}\\
    \frac{1-n}nq^{\nu(\varnothing)-2\nu(i)}, & \text{if $i=j$.}
  \end{cases}
  \]
Thus, it is enough to show that $B(\nu)$ is negative semidefinite for positive $q \le 1$. 
  Rescaling rows and columns of $B(\nu)$, we obtain a matrix $C(\nu)$  with entries
  \[
  C(\nu)_{ij} = \begin{cases}
    - (1-\frac{1}{n})+q^{-\nu(i,j)+\nu(i)+\nu(j)-\nu(\varnothing)}, & \text{if $i\neq j$,}\\
    - (1-\frac{1}{n}), & \text{if $i=j$.}
  \end{cases}
  \]
  By Lemma~\ref{lem:ultra}, 
$d(i,j)=2q^{-\nu(i,j)+\nu(i)+\nu(j)-\nu(\varnothing)}$ is an ultrametric on $E$. Since $d$ has radius $\le 1$ by the same lemma, Theorem~\ref{thm:psd} implies that $C(\nu)$ is negative semidefinite for positive $q \le 1$.
This implies the same for $B(\nu)$, and hence
 $A(\nu)$ 
has at most one positive eigenvalue for positive $q\le 1$. This finishes the proof that $Z_{q,\nu}$ is Lorentzian for positive $q \le 1$ when $\nu$ is \Mnat-concave.
\end{proof}

\begin{reptheorem}{thm:valuated}
For any \Mnat-concave function $\nu\colon 2^{E} \to \mathbb{R} \cup \{-\infty\}$,  we have
\[
I_{q,\nu;k}^2 \geq \Big(1+\frac1k\Big)\Big(1+\frac1{n-k}\Big) I_{q,\nu;k-1}
    I_{q,\nu;k+1} \ \ \text{for  all $0<k<n$ and $0<q\leq 1$.}
\]
\end{reptheorem}

\begin{proof}
Identifying the variables $x_i$ with each other in $Z_{q,\nu}(\x,y)$, we get a bivariate polynomial with coefficients $I_{q,\nu;k}$.
Since $Z_{q,\nu}$ is Lorentzian by Theorem~\ref{thm:qLorentzian}, the specialization is Lorentzian as well \cite[Theorem 2.10]{BH20}, and hence the sequence $I_{q,\nu;k}$ is ultra log-concave in $k$ for positive $q \le 1$ \cite[Example 2.3]{BH20}.
\end{proof}

\subsection{Polynomial log concavity}\label{sec:polynomial}

We prove Theorem~\ref{thm:qpolynomial} using the following observation.

\begin{lemma}\label{lem:coeffLorentzian}
Let $f$ be a homogeneous polynomial of degree $d$ in $n$ variables \[
    f(\x) = f_0(x_2,\dots,x_n) + x_1 f_1(x_2,\dots,x_n) + \dots + x_1^d f_d(x_2,\dots,x_n).
    \]
If $f$ is Lorentzian, then $f_i$ is Lorentzian for each $0\leq i \leq d$.
\end{lemma}

\begin{proof}
The polynomial $ \partial_1^i f$ is Lorentzian \cite[Corollary 2.11]{BH20}, and hence its specialization $i! f_i$ is Lorentzian as well  \cite[Theorem 2.10]{BH20}.
\end{proof}

For any $\nu\colon 2^{E}\to \R \cup \{-\infty\}$, nonnegative integers $a$ and $b$, and a multiset $S$ on $E$, we set
 \[
 N^{S}_{q,\nu}(a,b)  \coloneqq  \sum_{|A|=a,|B|=b, S=A+B} q^{-\nu(A)-\nu(B)},
 \]
 where the sum is over all $a$-element subset $A$ of $E$ and $b$-element subset $B$ of $E$ whose multiset union is $S$.

\begin{reptheorem}{thm:qpolynomial}
For any \Mnat-concave function $\nu:2^E \to \mathbb{R} \cup \{-\infty\}$, we have the coefficientwise inequality
  \[
  I_{q,\nu;k}(\x)^2 \succeq \Big(1+\frac1k\Big)I_{q,\nu;k-1}(\x) I_{q,\nu;k+1}(\x)
   \ \ \text{for  all $0<k<n$ and $0<q\leq 1$.}
  \]
\end{reptheorem}

\begin{proof}
Fix a positive real parameter $q \le 1$, and consider two copies of homogeneous generating polynomial  of an \Mnat-concave function $\nu$, say $Z_{q,\nu}(\x,y)$ and $Z_{q,\nu}(\x,z)$, where we use distinct homogenizing variables $y$ and $z$ for the same set of $n$ variables $\x$.
Since $Z_{q,\nu}(\x,y)$ and $Z_{q,\nu}(\x,z)$ are Lorentzian  by Theorem~\ref{thm:qLorentzian}, their product is Lorentzian too \cite[Corollary 2.32]{BH20}.
Writing $n$ for the cardinality of $E$ as before, we have
  \[
  Z_{q,\nu}(\x,y)Z_{q,\nu}(\x,z) = \sum\limits_{S} \left[\sum_{k+j=\abs{S}} N^S_{q,\nu}(j,k) y^{n-j}z^{n-k} \right] \x^S,
  \]
  where the sum is over all multisets $S$ on $E$ and $|S|$ is the cardinality of $S$ counting multiplicities. Iterating Lemma~\ref{lem:coeffLorentzian}, we see that, for each $S$, the bivariate polynomial
  \[
  \sum_{j+k=\abs{S}} N^S_{q,\nu}(j,k) y^{n-j}z^{n-k} 
  \]
  is Lorentzian.
  Therefore, the sequence of normalized coefficients $(n-j)! \ (n-k)! \ N^S_{q,\nu}(j,k)$ is log-concave and has no internal zeros \cite[Example 2.3]{BH20}. It follows that,  for all $i\leq j\leq k\leq l$ for with  $i+l=j+k=\abs{S}$, we have
  \begin{equation}\label{eq:lc4acoeff}
     (n-j)!\ (n-k)!\ N^S_{q,\nu}(j,k)  \ge (n-i)!\ (n-l)!\ N^S_{q,\nu}(i,l).
\end{equation}
  
This is not quite the inequality that we want, but we can apply it to a different  
 \Mnat-concave function $\nu'$ constructed from $\nu$ and $S$ to deduce the desired inequality
\[
j!\ k!\ N^S_{q,\nu}(j,k) \ge i!\ l!\ N^S_{q,\nu}(i,l).
\]
This will suffice, as the above displayed inequality for all $S$ is equivalent to the coefficient-wise inequality
\[
      j!\ k!\ I_{q,\nu;j}(\x)\ I_{q,\nu;k}(\x) \succeq
    i!\ l!\ I_{q,\nu;i}(\x)\ I_{q,\nu;l}(\x).
  \]

Let $S$ be a multiset on $E$ where each element of $E$ appears at most twice, and let $E'$ be the set of size
$n' := \abs{E'} = \abs{S}$
obtained from the underlying set $\underline{S}$ of $S$ by adding a second copy of each element that appears twice in $S$.
We define an \Mnat-concave function $\nu':2^{E'} \to \mathbb{R} \cup \{-\infty\}$ by setting
\[
\nu'(I)=\begin{cases} \nu(I), &\text{if $I$ is a subset of $\underline{S}$,}  \\ -\infty& \text{if otherwise.} \end{cases}
\]
By construction, for any nonnegative integers $a$ and $b$ satisfying $a+b=\abs{S}$, we have
\[
N^S_{q,\nu}(a,b)=N^S_{q,\nu'}(a,b).
\]
Applying \eqref{eq:lc4acoeff} to $\nu$ we obtain,   for all $i\leq j\leq k\leq l$ for with  $i+l=j+k=n'=\abs{S}$, that
  \[
   j!\ k!\ N^S_{q,\nu}(j,k)=(n'-j)!\ (n'-k)!\ N^{S'}_{q,\nu'}(j,k)  \ge (n'-i)!\ (n'-l)!\ N^{S'}_{q,\nu'}(i,l) = i!\ l!\ N^S_{q,\nu}(i,l)
  \]
as desired.  
\end{proof}

%\bibliography{trees}
\small

\bibliographystyle{amsalpha}
\providecommand{\bysame}{\leavevmode\hbox to3em{\hrulefill}\thinspace}
\providecommand{\MR}{\relax\ifhmode\unskip\space\fi MR }
\providecommand{\MRhref}[2]{%
  \href{http://www.ams.org/mathscinet-getitem?mr=#1}{#2}
}
\providecommand{\href}[2]{#2}

\end{document}